\documentclass[onefignum,onetabnum]{siamart190516}

\usepackage{url}
\usepackage{color}
\usepackage{graphicx}
\usepackage{amsmath,amssymb}
\usepackage{esint}
\usepackage[utf8]{inputenc}
\usepackage[english]{babel}
\usepackage{tikz}
\usetikzlibrary{arrows}

\numberwithin{equation}{section}
\numberwithin{theorem}{section}

\newsiamremark{remark}{Remark}

\newcommand{\ddiv}{\operatorname{div}}
\newcommand{\rot}{\operatorname{rot}}
\newcommand{\Curl}{\operatorname{Curl}}

\title{Computational lower bounds \\of the Maxwell eigenvalues}

\author{D. Gallistl\thanks{Friedrich-Schiller-Universit\"at Jena,
                    Institut für Mathematik,
                    Ernst-Abbe-Platz 2, 07743 Jena, Germany
                    (\email{dietmar.gallistl@uni-jena.de}).}
\and  V. Olkhovskiy\footnotemark[2]
}

\begin{document}

\maketitle

\begin{abstract}
A method to compute guaranteed lower
bounds to the eigenvalues of the Maxwell system 
in two or three space dimensions is proposed
as a generalization of the method of
Liu and Oishi [SIAM J. Numer. Anal., 51, 2013]
for the Laplace operator.
The main tool is the computation of an explicit
upper bound to the error of the Galerkin projection.
The error is split in two parts: one part is controlled
by a hypercircle principle and an auxiliary eigenvalue
problem. The second part requires a perturbation 
argument for the right-hand side replaced by
a suitable piecewise polynomial. The latter error
is controlled through the use of the commuting 
quasi-interpolation by Falk--Winther and
computational bounds on its stability constant.
This situation is different from the Laplace operator
where such a perturbation is easily controlled through
local Poincaré inequalities.
The practical viability of the approach is demonstrated
in test cases for two and three space dimensions.
\end{abstract}



\section{Introduction}

This paper is devoted to the computation of guaranteed lower
bounds of the Maxwell eigenvalues.
The Maxwell eigenvalue problem over a suitable bounded domain
$\Omega$ in dimension $d=2$ or $d=3$
seeks eigenpairs $(\lambda , u)$ with nontrivial $u$
such that
\begin{equation}\label{e:EVP_intro}
 (-1)^{d-1} \Curl\rot u = \lambda u \text{ in }\Omega
 \quad\text{and}\quad
 u \wedge \nu = 0 \text{ on }\partial\Omega.
\end{equation}
Here, $\nu$ is the outer unit normal to $\partial\Omega$
and $u\wedge \nu$ is the tangential trace of $u$.
The usual rotation (or curl) operator is denoted by
$\rot$, while its formal adjoint is denoted by $\Curl$;
precise definitions are given below.
The $\rot$ operator has an infinite-dimensional kernel
containing all admissible gradient fields,
leading to an eigenvalue $\lambda=0$ of infinite multiplicity.
Sorting out this eigenvalue in numerical computations
requires the incorporation of a divergence-free constraint.
In the setting of $\rot$-conforming finite elements,
the divergence constraint is necessarily imposed in 
a discrete weak form because simultaneous $\rot$ and
$\ddiv$-conformity may lead to non-dense and thus wrong
approximations \cite{Costabel1991,Monk}.
In conclusion, a variational form of \eqref{e:EVP_intro}
in an energy space $V$ is in general
approximated with a nonconforming
discrete space $V_h\not\subseteq V$
and no monotonicity principles are applicable for a
comparison of discrete and true eigenvalues.
Upper eigenvalue bounds can be expected from
discontinuous Galerkin (dG) schemes
\cite{BuffaPerugia2006}, but the practically more interesting
question of guaranteed lower bounds has remained open
until the contributions
\cite{BBB2014,BBB2017}.
For a detailed exposition of the eigenvalue problem
and its numerical approximation, the reader is
referred to
\cite{Hiptmair2002,Monk,Boffi2010}
and the references therein.

In the finite element framework for coercive operators
in some Hilbert space $V$
(e.g., the Laplacian), guaranteed lower 
eigenvalue bounds were successfully derived by the 
independent contributions 
\cite{LiuOishi} and \cite{CarstensenGedicke},
which basically follow the same reasoning, illustrated
here for the first eigenpair $(\lambda, u)$
of a variational eigenvalue problem
$$
  a(u,v) = \lambda b(u,v) \quad\text{for all }v\in V
$$
with inner products $a$ and $b$ and corresponding
norms $\|\cdot\|_a$ and $\|\cdot\|_b$.
For a (possibly nonconforming) discretization $V_h$
with the first discrete eigenpair $(\lambda_h,u_h)$,
the discrete Rayleigh--Ritz principle 
\cite{WeinsteinStenger1972}
implies
$$
  \lambda_h \|v_h\|_b^2 \leq \|v_h\|_a^2
$$
for any $v_h\in V_h$.
Given an $a$-orthogonal projection operator $G_h$
(assuming $a$ is defined on the sum $V+V_h$),
this and some elementary algebraic manipulations
show
$$
  \lambda_h \|G_h u\|_b^2 
  \leq \|G_h u\|_a^2
  \leq \|u\|_a^2 - \|u-G_h u\|_a^2 .
$$
Assuming the normalization $\|u\|_b=1$ so that
$\|u\|_a^2=\lambda$, it turns out that explicit control
of $\|u-G_h u\|_b$ by $\|u-G_h u\|_a$ yields a computational
lower bound.
In \cite{LiuOishi} $G_h$ is the standard Galerkin
projection while in \cite{CarstensenGedicke} is the 
interpolation operator in a Crouzeix--Raviart method.
Further approaches to the computation of lower eigenvalue bounds
were provided by
\cite{SebestovaVejchodsky2014,CancesDussonMadayStammVohralik2018}.

In this paper we aim at extending the idea of 
\cite{LiuOishi} to the Maxwell eigenvalue problem
\eqref{e:EVP_intro} discretized with lowest-order
N\'ed\'elec (edge) elements \cite{Monk}.
The main novelty in contrast to \cite{LiuOishi} 
is the guaranteed computational control of the Galerkin
error in a linear Maxwell system with right-hand side $f$.
In general, the estimate takes the format 
$$\|u-u_{h}\|_{a} \leq M_{h} \|f\|_{b}$$
with a mesh-dependent number $M_{h}$,
for which we propose a computational upper bound in this
paper.
In \cite{LiuOishi} such bound is achieved for the Laplacian
by splitting $f$ in a piecewise polynomial part $f_h$ and
some remainder. The first part of the error is quantified
through a hypercircle principle \cite{Braess2007}
and an auxiliary global eigenvalue problem.
This idea goes back to 
the work \cite{KikuchiSaito2007} on a~posteriori error
estimators and was used in the context of eigenvalue
problems by
\cite{LiuOishi,TanakaTakayasuLiuOishi2014,Liu2015,Liu2021}.
The second part of the error is ---in the case of the Laplacian,
where $f_h$ is simply the piecewise mean of $f$---
easily controlled because it reduces to element-wise Poincar\'e
inequalities whose constants can be explicitly bounded
\cite{PayneWeinberger}.
In the present case of the Maxwell system, the situation is more
involved. The $\rot$ operator maps the N\'ed\'elec space
to the divergence-free Raviart--Thomas elements
\cite{BoffiBrezziFortin2013},
and the $L^2$-orthogonal projection to the latter
is nonlocal and explicit bounds on that projection are unknown.
In order to obtain a computable bound, we make use of
recent developments of Finite Element Exterior Calculus
\cite{ArnoldFalkWinther2006,Arnold2019},
namely the Falk--Winther projection \cite{FalkWint1}.
This family of operators commutes with the exterior derivative
and is locally defined, so that it is actually computable.
Practical implementations of the operator have been used
in the context of numerical homogenization
\cite{GallistlHenningVerfurth,HenningPerssonCode,HenningPersson}.
In this work, the advantage of the local construction is 
that the involved stability constant can be computationally
bounded from above. In a perturbation argument for
the Maxwell system, this tool replaces the Poincar\'e inequality
from the Laplacian case.
The bounds are achieved by solving local discrete
eigenvalue problems combined with standard estimates.

The main result
is a computabe upper bound $\hat M_h$ to $M_{h}$, which
results in the guaranteed lower bound 
$$
   \frac{\lambda_h}{1+\hat M_h^2\lambda_h} \leq \lambda 
$$
from Theorem~\ref{t:lowerbound}
for the $k$th Maxwell eigenvalue $\lambda$.
The quantities on the left-hand side are the $k$th discrete
eigenvalue  $\lambda_h$ and the computable 
mesh-dependent number $\hat M_h$.
In particular, the computation of $\hat M_h$ involves
guaranteed control over the bound for the 
Galerkin error in a linear Maxwell problem
and the local stability constants of the 
Falk--Winther interpolation.
The guaranteed computational bound $\hat M_h$ for $M_{h}$
is carefully described in this paper.
The mesh-dependent quantity $M_h$ is required to be uniform with respect to the
right-hand side $f$ and is therefore related to
elliptic regularity of the linear Maxwell problem on the 
specific domain of interest.
On polytopal domains it is expected to scale like
the power $h^s$ of the maximum mesh size $h$ with
some exponent $0<s\leq 1$.
This limits efficient computations to the case of
lowest-order N\'ed\'elec (edge) elements.
Such limitation is also encountered in the existing
works \cite{CarstensenGedicke,LiuOishi,Liu2015}
for the Laplacian.

\medskip
The remaining parts of this article are organized as follows.
Section~\ref{s:prelim} lists preliminaries on the 
Maxwell problem, discrete spaces, and the Falk--Winther
interpolation.
Guaranteed computational bounds on the Galerkin
error are presented in Section~\ref{s:Galerkin}.
The lower eigenvalue bounds are shown in
Section~\ref{s:eigenvaluebounds};
the practical computation of the relevant constants
is described in Section~\ref{s:constants};
and actual computations are shown
in the numerical experiments of Section~\ref{s:num}.
The remarks of Section~\ref{s:conclusion} conclude this paper.

\section{Preliminaries}\label{s:prelim}

\subsection{Notation}

Let $\Omega\subseteq\mathbb R^d$ for $d\in\{2,3\}$ be a
bounded and open polytopal Lipschitz domain,
which we assume to be 
contractible.
The involved differential operators read
\begin{equation*}
\rot v = \partial_1 v_2 - \partial_2 v_1
\text{ for } d=2
\quad\text{and}\quad
\rot v =
 \left(
 \begin{matrix}
  \partial_2 v_3-\partial_3 v_2 \\
  \partial_3 v_1 - \partial_1 v_3\\
  \partial_1 v_2 - \partial_2 v_1
 \end{matrix}
 \right)
\text{ for } d=3.
\end{equation*}
For the formal adjoint operators we write
\begin{equation*}
\Curl \phi 
  =
    \left(
     \begin{matrix}
        -\partial_2 \phi \\ \partial_1 \phi
     \end{matrix}
 \right)
\text{ for } d=2
\quad\text{and}\quad
\Curl = \rot
\text{ for } d=3
\end{equation*}
($\phi$ is a scalar function for $d=2$)
so that the integration-by-parts formula
$$
  \int_\Omega \phi \rot v \,dx
  =
  (-1)^{d-1} \int_\Omega \Curl \phi \cdot v \,dx  
$$
holds for sufficiently regular 
scalar functions ($d=2$) or vector fields ($d=3$)
$\phi$ and vector fields $v$ 
with vanishing tangential trace over
$\partial\Omega$.

Standard notation on Lebesgue and Sobolev spaces is employed
throughout this paper. Given any open set
$\omega\subseteq \mathbb R^d$, 
the $L^2(\omega)$ inner product is denoted by 
$(\cdot,\cdot)_{L^2(\omega)}$ with the norm
$\|\cdot\|_{L^2(\omega)}$.
The usual $L^2$-based first-order Sobolev space is denoted
by $H^1(\Omega)$
and $H^1_0(\Omega)$ is the subspace with vanishing
trace over $\partial\Omega$.
The space of $L^2$ vector fields over $\Omega$ with
weak divergence in $L^2(\Omega)$ is denoted by
$H(\ddiv,\Omega)$; and the subspace of divergence-free
vector fields reads $H(\ddiv^0,\Omega)$.
The space of $L^2(\Omega)$ vector fields with weak
rotation in $L^2(\Omega)$ is denoted by
$H(\rot,\Omega)$ while its subspace with vanishing
tangential trace is denoted by 
$H_0(\rot,\Omega)$.

In the context of eigenvalue problems,
the $L^2$ inner product is also denoted by $b(\cdot,\cdot)$
and the $L^2$ norm is denoted by $\|\cdot\|_b$.

On $H_0(\rot,\Omega)$, we define the bilinear form
$$
  a(v,w) := (\rot v, \rot w)_{L^2(\Omega)}
  \quad\text{for any }v,w\in H_0(\rot,\Omega).
$$
Let 
$V:= H_0(\rot,\Omega)\cap H(\ddiv^0,\Omega)$.
On $V$, the form $a$ is an inner product 
\cite[Corollary 4.8]{Monk}
and the seminorm
$\|\cdot\|_a = \sqrt{a(\cdot,\cdot)}$ is a norm on $V$.
Given a divergence-free right-hand side $f\in H(\ddiv^0,\Omega)$,
the linear Maxwell problem seeks 
$u\in V $ such that
\begin{equation}\label{e:maxwell_linear}
 a(u,v) = b(f,v)
 \quad\text{for all }v\in V.
\end{equation}
It is well known \cite{Monk}
and needed in some arguments of this article
that \eqref{e:maxwell_linear} is even satisfied for all test functions
$v$ from the larger space $H_0(\rot,\Omega)$.

Let $\mathcal T$ be a regular simplicial triangulation of
$\Omega$.
The diameter of any $T\in\mathcal T$ is denoted by
$h_T$ and $h_{\max} := \max_{T\in\mathcal T} h_T$ is the 
maximum mesh size.
Given any $T\in\mathcal T$,
the space of first-order polynomial functions over $T$
is denoted by $P_1(T)$.
The lowest-order standard finite element space 
(with or without homogeneous
Dirichlet boundary conditions) is denoted by
$$
S^1(\mathcal T)
:=
\{v\in H^1(\Omega): 
           \forall T\in\mathcal T, v|_T\in P_1(T)
           \}
\text{ and }
S^1_0(\mathcal T) := H^1_0(\Omega)\cap S^1(\mathcal T).
$$
The space of lowest-order edge elements 
\cite{Monk,BoffiBrezziFortin2013}
reads
$$
\begin{aligned}
\mathcal N_{0}(\mathcal T) =
\{ v\in H(\rot,\Omega):
  \forall T\in\mathcal T\;
   \exists \alpha_T \in \mathbb R^d\;
   \exists \beta_T \in P_1(T)^d
   \forall x\in T:
   \\
  v(x) = \alpha_T + \beta_T(x)
  \text{ and } \beta_T(x)\cdot x = 0
\}
\end{aligned}
$$
and we denote
$$
\mathcal N_{0,D}(\mathcal T)
 := \mathcal N_{0}(\mathcal T) \cap H_0(\rot,\Omega)  .
$$
The approximation of \eqref{e:maxwell_linear} with
edge elements uses the space
\begin{equation}
\label{e:Vhdef}
 V_h = \{\psi_h\in \mathcal N_{0,D}(\mathcal T) 
           : \forall v_h\in S^1_0(\mathcal T)\;
                             (\nabla v_h,\psi_h)_{L^2(\Omega)} = 0\}.
\end{equation}
The elements of $V_h$ are weakly divergence-free and need in general
not be elements of $V$, i.e., $V_h\not\subseteq V$.
It is known \cite{Monk}
that $a$ is an inner product on $V_h$.
The finite element system seeks $u_h\in V_h$ such that
\begin{equation}\label{e:maxwell_linear_discr}
 a(u_h, v_h) = b(f,v_h)
 \quad\text{for all }v_h\in V_h.
\end{equation}
We remark that, in practical computations, systems like 
\eqref{e:maxwell_linear_discr} are solved as mixed systems
with a Lagrange
multiplier enforcing the linear constraint in \eqref{e:Vhdef}.
Given $u$, its approximation $u_h=G_h u$ is called the
Galerkin projection.
This terminology is justified by the fact that the 
approach is conforming when viewed in a saddle-point
setting.
In particluar, since \eqref{e:maxwell_linear}
is satisfied for all $v\in H_0(\rot,\Omega)$,
the following ``Galerkin orthogonality''
is valid
\begin{equation}\label{e:galerkinorthogonality}
 a(u-u_h,v_h) = 0 \quad\text{for all }v_h\in V_h.
\end{equation}
The Raviart--Thomas
finite element space \cite{BoffiBrezziFortin2013}
is defined as
$$
\begin{aligned}
\mathit{RT}_0(\mathcal T)
:=
\{ v\in H(\ddiv,\Omega):
  &
  \forall T\in\mathcal T \exists (\alpha_T,\beta_T)\in\mathbb R^d\times \mathbb R
   \\
   &\forall x\in T, v|_T(x) = \alpha_T + \beta_T x
 \} .
\end{aligned}
$$

\subsection{Falk--Winther interpolation}\label{ss:FW}

Given a regular triangulation $\mathcal T$ 
and any element $T\in\mathcal T$, the element patch
built by the simplices having nontrivial intersection
with $T$ is defined as
$$
 \omega_T := 
 \operatorname{int} 
 (\cup \{K\in\mathcal T: K\cap T\neq \emptyset\} )
 .
$$
There is a projection
$\pi^{\ddiv}: H(\ddiv,\Omega)\to \mathit{RT}_0(\mathcal T)$
with local stability in the sense that there exist
constants $C_{1,\ddiv}$, $C_{2,\ddiv}$ such that
for any $T\in\mathcal T$ and any $v\in H(\ddiv,\Omega)$
we have
\begin{equation}\label{e:fwdivstab}
 \|\pi^{\ddiv} v\|_{L^2(T)}
 \leq
 C_{1,\ddiv} \|v\|_{L^2(\omega_T)}
 + h_T C_{2,\ddiv}  \|\ddiv v\|_{L^2(\omega_T)} .
\end{equation}
Furthermore,
there is a projection
$\pi^{\Curl}: H(\Curl,\Omega)\to \mathcal S_h$
where 
\begin{equation*}
 H(\Curl,\Omega) =
    \begin{cases}
       H^1(\Omega;\mathbb R^2) &\text{if }d=2 \\
       H(\rot,\Omega)          &\text{if }d=3
    \end{cases}
\quad\text{and}\quad
  \mathcal S_h =
     \begin{cases}
      S^1(\mathcal T) &\text{if }d=2 \\
      \mathcal N_0(\mathcal T)    &\text{if }d=3
     \end{cases}
\end{equation*}
with constants $C_{1,\Curl}$, $C_{2,\Curl}$ such that
for any $T\in\mathcal T$ and any $v\in H(\Curl,\Omega)$
we have
\begin{equation}\label{e:fwrotstab}
 \|\pi^{\Curl} v\|_{L^2(T)}
 \leq
 C_{1,\Curl} \|v\|_{L^2(\omega_T)}
 + h_T C_{2,\Curl} \|\Curl v\|_{L^2(\omega_T)} .
\end{equation}
The crucial property is that these two operators
commute with the exterior derivative in the sense that
$ \pi^{\ddiv}\Curl = \Curl\pi^{\Curl}$.
The corresponding commuting diagram is displayed in
Figure~\ref{f:fwoperator}.
For the construction of the operators
$\pi^{\ddiv}$ and $\pi^{\Curl}$, the reader is referred
to \cite{FalkWint1} and Section~\ref{ss:FWcomputation} below.

\begin{figure}
\begin{center}
\begin{tikzpicture}[node distance=3.0cm]
\node (Hrot) {$H(\Curl,\Omega)$};
\node (Hdiv)[right of=Hrot]{$H(\ddiv,\Omega)$};
\node (Nedelec)[below of= Hrot]{$\mathcal S_h$};
\node (RT)[right of=Nedelec]{$\mathit{RT}_0(\mathcal T)$};
\draw[transform canvas={yshift=0.5ex},->] (Hrot) --(Hdiv) node[above,midway] {\footnotesize $\Curl$};
\draw[transform canvas={yshift=0.5ex},->] (Nedelec)--(RT) node[above,midway] {\footnotesize $\Curl$};
\draw[transform canvas={xshift=0.5ex},->] (Hrot) --(Nedelec) node[right,midway] {\footnotesize $\pi^{\Curl}$};
\draw[transform canvas={xshift=0.5ex},->] (Hdiv) --(RT) node[right,midway] {\footnotesize $\pi^{\ddiv}$};
\end{tikzpicture}
\end{center}
\caption{Commuting diagram of the Falk--Winther operator.
        \label{f:fwoperator}
        }

\end{figure}

\section{Bounds on the Galerkin projection}\label{s:Galerkin}

The goal of this section is a fully computable bound
on the Galerkin error.

\subsection{\texorpdfstring{$L^2$}{L2} error control}

From elliptic regularity theory
(see
\cite[Theorem 3.50]{Monk}
and \cite{CostabelDauge2000}), it is known that
the solution $u\in V$ to \eqref{e:maxwell_linear} satisfies
$$
\|u\|_{H^s(\Omega)} + \|\rot u\|_{H^s(\Omega)} \leq C \|f\|_b
$$
for some positive $s>1/2$,
where $\|\cdot\|_{H^s(\Omega)}$ is the usual
fractional-order Sobolev space \cite{Monk}.
The Galerkin property~\eqref{e:galerkinorthogonality}
and well-known interpolation error estimates 
\cite[Theorem 5.25]{Monk}
show that
\begin{equation}
\label{e:apriori}
\|u-G_h u\|_a 
\leq
\inf_{v_h\in V_h}\|u-v_h\|_a 
\leq
C h^s \|\rot u\|_{H^s(\Omega)}.
\end{equation}
Hence,
there
exists a mesh-dependent (but $f$-independent) number $M_h$ such that
\begin{equation}\label{e:Galerkin_schranke}
  \|u-G_h u\|_a \leq M_h \|f\|_b,
\end{equation}
where it is understood that $M_{h}$ is the optimal choice
(uniformly in $\|f\|_b$). 
On convex domains, $M_h$ is proportional to the mesh size $h$
while in general, reduced regularity
implies that $M_h$ is proportional
to $h^s$ with some $0<s\leq1$.
Theorem~\ref{t:Mhtheorem} below states a 
computable upper bound $M_{h} \leq \hat M_h$.

Given some $u\in V$, the Galerkin approximation
$G_h u$ is usually not divergence-free and therefore
possesses a nontrivial $L^2$ orthogonal decomposition
\begin{equation}\label{e:helmholtz}
    G_h u = \nabla \phi + R
\end{equation}
with $\phi \in H^1_0(\Omega)$ and $R\in H(\ddiv^0,\Omega)$.
The inclusion $G_h u\in H_0(\rot,\Omega)$ furthermore
shows that $R\in V$.
The following lemma states an $L^2$ error estimate.
The proof uses the classical Aubin--Nitsche duality technique.

\begin{lemma}\label{l:L2}
The divergence-free part $u-R$ of the error $u-G_h u$ satisfies
the following error estimate
 $$
   \| u - R \|_b \leq M_h \| u - G_h u \|_a  .
 $$ 
\end{lemma}
\begin{proof}
 The error $e:=u-R$ is divergence-free and thus
 $e\in V$.
 There exists
 a unique solution $z\in V$ satisfying
 $$
    a(z,v) = b(e,v) \quad\text{for all }v\in V.
 $$
Since $e\in V$, we infer with the symmetry of $a$ 
and $\rot\nabla = 0$ that

$$ 
  \|e\|_b^2 = b(e,e) = a(z, e) = a(e,z) = a(u-G_h u, z) .
$$
The Galerkin property 
\eqref{e:galerkinorthogonality}
shows that $u-G_h u$ is $a$-orthogonal
to any element of $V_h$. Thus
$$ 
  \|e\|_b^2 
 = a(u-G_h u, z -G_h z) 
 \leq \| u-G_h u\|_a \; \| z-G_h z\|_a . 
$$
The application of \eqref{e:Galerkin_schranke} to
$z$ with right-hand side $e$ reveals that
the Galerkin error $\| z-G_h z\|_a$
is bounded by $M_h \| e\|_b$,
which implies the asserted bound.
\end{proof}

\subsection{Perturbation of the right-hand side}

In this section, we quantify the error that arises in the
solution of the linear Maxwell system when
the right-hand side $f$ is replaced by a piecewise polynomial approximation $f_h$.

The regular decomposition
\cite[Proposition 4.1]{CostabelMcIntosh}
states that there exists a constant $C_{RD}$ such for that every 
$f\in H(\ddiv^0,\Omega)$
there exists
$\beta\in H^1(\Omega;\mathbb R^{2d-3})$ such that 
\begin{equation}\label{e:regulardecomp}
\Curl \beta = f
\qquad \text{and}\qquad
\|D \beta\|_{L^2(\Omega)} \leq C_{RD} \|f\|_{L^2(\Omega)} 
\end{equation}
where $D\beta$ denotes the derivative (Jacobian matrix)
of the vector field $\beta$.
We remark that in the two-dimensional case
the field $\Curl\beta$ is a rotation of $D\beta$
so that $C_{RD}=1$ if $d=2$.

Given $T\in\mathcal T$ and its element patch
$\omega_T$, the Poincar\'e inequality states for any
$H^1$ function $v$ with vanishing average over $\omega_T$
that $\|v\|_{L^2(\omega_T)}\leq C(T) \|D v\|_{L^2(\omega_T)}$
with a constant $C(T)$ proportional to $h_T$.
By $\tilde c$ we denote the smallest constant such that
$$
  \|v\|_{L^2(\omega_T)}\leq h_T \tilde c \|D v\|_{L^2(\omega_T)}
$$
holds for all such functions uniformly in $T\in\mathcal T$.

Recall that, due to its commutation property,
the Falk--Winther interpolation $\pi^{\ddiv}$
maps $H(\ddiv^0,\Omega)$ to
$\mathit{RT}_0(\mathcal T)\cap H(\ddiv^0,\Omega)$.
We denote the overlap constant of element patches
by
$$
C_{\mathit{OL}} 
  = 
  \max_{T\in\mathcal T} 
  \operatorname{card}\{K\in\mathcal T : 
          T\subseteq   \overline {\omega_K} \}.
$$

\begin{lemma}\label{l:perturb}
Let $f\in H(\ddiv^0,\Omega)$ and let 
$
 \tilde f 
 = \pi^{\ddiv} f \in \mathit{RT}_0(\mathcal T)\cap H(\ddiv^0,\Omega)
$
be its Falk--Winther interpolation.
Let $u$ and $\tilde u$ denote the solution to
\eqref{e:maxwell_linear} with right-hand side $f$ and $\tilde f$,
respectively.
Then
\begin{equation*}
\|u-\tilde u\|_a
 \leq
 \sqrt{C_{\mathit{OL}} }
h_{\max}
\hat C
 \|f \|_{L^2(\Omega)} 
\end{equation*}
for the constant
\begin{align*}
\hat C :=
\begin{cases}
(1+C_{1,\Curl}) \tilde c + C_{2,\Curl}
 & \text{if }d=2, \\
\sqrt{2(((1+C_{1,\Curl}) \tilde c C_{\mathit{RD}})^2    
 + C_{2,\Curl}^2) }
  &\text{if }d=3.
\end{cases}
\end{align*}
\end{lemma}
\begin{proof}
 Abbreviate $e:=u-\tilde u$.
 The solution properties imply
 $$
   \| u-\tilde u\|_a^2 = a(u-\tilde u,e) 
   =
   b(f-\tilde f,e)
   =
   b(f-\pi^{\ddiv} f,e) .
 $$
From the regular decomposition
\eqref{e:regulardecomp}
and the commuting property
of the operators $\pi^{\ddiv}$, $\pi^{\Curl}$ we obtain
$$
 f - \pi^{\ddiv} f 
 = \Curl \beta-\pi^{\ddiv} \Curl \beta
 = \Curl (\beta - \pi^{\Curl} \beta).
$$
Thus, integration by parts and the homogeneous boundary
conditions of $e$ imply
\begin{align*}
b(f-\pi^{\ddiv} f,e)
&=
b(\Curl (\beta - \pi^{\Curl} \beta),e)
\\
&=
(-1)^{d-1}
b(\beta - \pi^{\Curl} \beta,\rot e)
\leq 
\|\beta - \pi^{\Curl} \beta \|_b \|e\|_a.
\end{align*}
The combination with the above chain of identities
implies
\begin{equation}
\label{e:perturb_a}
\| u-\tilde u\|_a \leq \|\beta - \pi^{\Curl} \beta \|_b .
\end{equation}

Let $T\in\mathcal T$ be arbitrary.
Since $\pi^{\Curl}$ locally preserves constants, we obtain for 
the patch average
$$
    \bar\beta := \fint_{\omega_T} \beta \,dx
$$
that the difference
$\beta - \pi^{\Curl} \beta$ can be split with
the triangle inequality and 
the inclusion $T\subseteq \overline{\omega_T}$ 
as follows
\begin{align*}
 \| \beta - \pi^{\Curl} \beta\|_{L^2(T)}
 \leq \| \beta - \bar{\beta}\|_{L^2(\omega_T)}
           + \|\pi^{\Curl}(\beta - \bar{\beta})\|_{L^2(T)} .
 \end{align*}
 Estimate \eqref{e:fwrotstab} with $\Curl \beta = f$
 followed by the Poincar\'e inequality on
 $\omega_T$ with constant $h_T\tilde c$ 
 thus reveal
 \begin{align*}
\| \beta - \pi^{\Curl} \beta\|_{L^2(T)}
\leq 
  h_T \tilde c(1+C_{1,\Curl}) 
   \|D \beta \|_{L^2(\omega_T)}
 + C_{2,\Curl} h_T \|f \|_{L^2(\omega_T)} .
\end{align*}
This local result generalizes
to the whole domain $\Omega$ as follows
\begin{align}
\begin{aligned}\label{e:lemmaproofformel}
\| \beta - \pi^{\Curl} \beta\|_b^2
&=
\sum_{T \in \mathcal{T}} \| \beta - \pi^{\Curl} \beta\|_{L^2(T)}^2
\\
&\leq 
\sum_{T \in \mathcal{T}} 
(
h_T \tilde c(1+C_{1,\Curl}) 
   \|D\beta \|_{L^2(\omega_T)}
 + C_{2,\Curl} h_T \|f \|_{L^2(\omega_T)} 
 )^2
.
\end{aligned}
\end{align}
If $d=2$ we have the identity
$\|D\beta \|_{L^2(\omega_T)}=\|f \|_{L^2(\omega_T)} $
and therefore conclude
\begin{align*}
\| \beta - \pi^{\Curl} \beta\|_b^2
&\leq 
\sum_{T \in \mathcal{T}}
h_T^2 
(\tilde c(1+C_{1,\Curl}) + C_{2,\Curl})^2
 \|f \|_{L^2(\omega_T)}^2
\\
&\leq
C_{\mathit{OL}} h_{\max}^2 
(\tilde c(1+C_{1,\Curl}) + C_{2,\Curl})^2 \|f \|_{L^2(\Omega)}^2
  .
\end{align*}
The squared expression in parentheses on the right-hand
side equals $\hat C$ if $d=2$, whence the assertion follows in that case.
If $d=3$, we again use \eqref{e:lemmaproofformel}
and compute
\begin{align*}
\| \beta - \pi^{\Curl} \beta\|_b^2
&\leq 
2
\sum_{T \in \mathcal{T}} h_T^2
\left(
   (\tilde c(1+C_{1,\Curl}) 
   \|D\beta \|_{L^2(\omega_T)}  )^2
 + (C_{2,\Curl}  \|f \|_{L^2(\omega_T)} )^2
 \right)
\\ 
&\leq
2 C_{\mathit{OL}} h_{\max}^2 
\left(
(\tilde c(1+C_{1,\Curl}) 
   \|D \beta \|_{L^2(\Omega)}  )^2
 + (C_{2,\Curl} \|f \|_{L^2(\Omega)} )^2
 \right) .
\end{align*}
We estimate $\|D \beta \|_{L^2(\Omega)}$
with $\eqref{e:regulardecomp}$
and obtain
\begin{align*}
\| \beta - \pi^{\Curl} \beta\|_b
\leq
\sqrt{2 C_{\mathit{OL}} }
h_{\max}
\hat C
 \|f \|_{L^2(\Omega)} .
\end{align*}
The combination with \eqref{e:perturb_a} concludes the proof.
\end{proof}

\subsection{Error bound on the Galerkin projection}
\label{ss:galerkinerror}
We denote the space of divergence-free Raviart--Thomas functions
by $X_h := \mathit{RT}_0(\mathcal T)\cap H(\ddiv^0,\Omega)$.
Given $f_h\in X_h$, let
$$
 \mathcal S_{f_h} := 
  \{v_h\in \mathcal S_h : \Curl v_h = (-1)^{d-1}f_h\}.
$$
Define
\begin{equation}\label{e:kappahdef}
\kappa_h :=
\max_{f_h\in X_h\setminus\{0\}}
\min_{v_h\in V_h}
 \min_{\tau_h\in \mathcal S_{f_h}}
 \frac{\|\tau_h - \rot v_h\|_{L^2(\Omega)}}{\|f_h\|_{L^2(\Omega)}}.
\end{equation}

The next lemma states that, for discrete data,
the Galerkin error can be quantified through
$\kappa_h$.

\begin{lemma}\label{l:kappah}
Let $f_h \in X_h$, 
and let 
$\tilde u \in V$
and $\tilde u_h\in V_h$
be solutions to
\eqref{e:maxwell_linear} and \eqref{e:maxwell_linear_discr},
respectively, with right-hand side $f=f_h$.
Then the following computable error estimate holds
$$ 
 \| \tilde u-\tilde u_h\|_a \leq \kappa_{h} \|f_h\|_b .
$$
\end{lemma}

\begin{proof}
Let $v_h\in V_h$ be arbitrary
and let $\tau_h\in\mathcal S_{f_h}$.
Integration by parts implies the orthogonality
$$
 (\tau_h - \rot \tilde u,
  \rot (\tilde u-v_h) )_{L^2(\Omega)} = 0.
$$
Thus the following hypercircle identity holds
\begin{equation}\label{e:hypercircle}
 \|\tau_h - \rot \tilde u\|_{L^2(\Omega)}^2 
  + 
 \|\rot (\tilde u-v_h) \|_{L^2(\Omega)}^2
 =
 \|\tau_h - \rot v_h\|_{L^2(\Omega)}^2.
\end{equation}
Thus,
$$
\| \tilde u-v_h \|_a
=
\| \rot (\tilde u-v_h) \|_{L^2(\Omega)}
 \leq 
 \min_{\tau\in \mathcal S_{f_h}}
 \|\tau_h - \rot v_h\|_{L^2(\Omega)}.
$$
The left-hand side is minimal for 
$\tilde u_h$ amongst all elements
$v_h\in V_h$, whence
$$
\| \tilde u-\tilde u_h \|_a
 \leq 
 \min_{v_h\in V_h}
 \min_{\tau\in \mathcal S_{f_h}}
 \|\tau_h - \rot v_h\|_{L^2(\Omega)}
 \leq \kappa_h \|f_h\|_{L^2(\Omega)}
$$
where the last estimate follows
from the definition \eqref{e:kappahdef}.
\end{proof}

\begin{remark}
 Hypercircle (or Prager--Synge) identities like \eqref{e:hypercircle}
 are often used in the context of a~posteriori error estimation
 \cite{Braess2007,Vejchodsky2006,Vejchodsky2012,ChaumontErnVohralik2021}.
\end{remark}

\begin{remark}
 The asymptotic convergence rate of $\kappa_h$ 
 is that of the sum of the primal and dual mixed
 finite element error.
 This can be seen from taking
 the minimum over $v_h$ and $\tau_h$
 in the hypercircle relation \eqref{e:hypercircle}
 for the ``worst'' $L^2$-normalized right-hand side
 $f_h$.
 Thus, $\kappa_h$ is proportional to $h^s$ with the 
 elliptic regularity index $s$ from \eqref{e:apriori}.
\end{remark}

For possibly non-discrete $f$,
the Galerkin error is bounded by the 
following perturbation argument.

\begin{theorem}\label{t:Mhtheorem}
Given $f\in H(\ddiv^0,\Omega)$,
let $u\in V$ solve \eqref{e:maxwell_linear}
and let $u_h\in V_h$ solve \eqref{e:maxwell_linear_discr}.
Let $\hat M_h$ with
$$
\hat M_h\geq (h_{\max}\hat C
   +
   \kappa_h C_{1,\ddiv}
   )
   \sqrt{ C_{\mathit{OL}}}
$$
(with $\hat C$ from Lemma~\ref{l:perturb})
be given.
Then, the following error bound is satisfied 
 $$
  \| u-u_h \|_a
  \leq \hat M_h \|f\|_{L^2(\Omega)} .
 $$
In particular, $M_{h} \leq \hat M_h$.
\end{theorem}
\begin{proof}
Let $f_h=\pi^{\ddiv}f$ denote the 
Falk--Winther interpolation of $f$
and denote as in Lemma~\ref{l:kappah}
by $\tilde u$ and $\tilde u_h$
the solutions with respect to $f_h$.
Since $u_h\in V_h$ minimizes the error
$\|u-v_h\|_a$ amongst all $v_h\in V_h$,
we have
$\|u-u_h\|_a\leq\|u-\tilde u_h\|_a$.
 The triangle inequality then leads to
 \begin{equation}\label{e:Mhlemma_a}
   \| u-u_h \|_a
   \leq 
   \| u-\tilde u \|_a 
    + \| \tilde u-\tilde u_h \|_a .
 \end{equation}
 The second term on the right-hand side is bounded through
 Lemma~\ref{l:kappah} as follows
 \begin{equation*}
   \| \tilde u-\tilde u_h \|_a 
   \leq
   \kappa_h \| f_h\|_{L^2(\Omega)}
   = \kappa_h \| \pi^{\ddiv} f\|_{L^2(\Omega)} .
 \end{equation*}
 The local stability \eqref{e:fwdivstab} of $\pi^{\ddiv}$
 (note that $\ddiv f=0$)
 and the overlap of element patches
 imply
 $$
   \| \pi^{\ddiv} f\|_{L^2(\Omega)}^2
   =
   \sum_{T\in\mathcal T} \| \pi^{\ddiv} f\|_{L^2(T)}^2
   \leq 
   \sum_{T\in\mathcal T} 
     C_{1,\ddiv}^2 \| f\|_{L^2(\omega_T)}^2
   \leq
    C_{\mathit{OL}} C_{1,\ddiv}^{2} \|f\|_{L^2(\Omega)}^2 .
 $$
 Thus,
 \begin{equation}\label{e:Mhlemma_b}
  \| \tilde u-\tilde u_h \|_a
   \leq 
   \kappa_h \sqrt{ C_{\mathit{OL}}}C_{1,\ddiv}\|f\|_{L^2(\Omega)} .
 \end{equation}
 The first term on the right-hand side
 of \eqref{e:Mhlemma_a}
 is bounded through Lemma~\ref{l:perturb}
 by $\sqrt{C_{\mathit{OL}}} h_{\max}\hat C$.
 The combination of this bound with
 \eqref{e:Mhlemma_a} and \eqref{e:Mhlemma_b} concludes the proof.
\end{proof}

The foregoing theorem shows that a computable upper bound
$\hat M_h$ to $M_h$ can be found if 
values or upper bounds
of $\hat C$, $\kappa_h$, $C_{1,\ddiv}$ are
available.
Their computation is described in Section~\ref{s:constants}.

\section{Eigenvalue problem and lower bound}
\label{s:eigenvaluebounds}
The Maxwell eigenvalue problem seeks pairs
$(\lambda,u)\in \mathbb R\times V$ with
$\|u\|_b=1$ such that
\begin{equation}\label{e:maxwell_EVP}
 a(u,v) = \lambda b(u,v) 
 \quad\text{for all }v\in V.
\end{equation}
The condition $u\in V$ implies that $\ddiv u=0$
and, thus, the non-compact part of the spectrum
of the $\Curl\rot$ operator (corresponding to gradient
fields as eigenfunctions) is sorted out in this formulation.
It is well known \cite{Monk} that the eigenvalues to
\eqref{e:maxwell_EVP} form an infinite discrete set
$$
  0<\lambda_1\leq \lambda_2\leq \dots 
  \quad\text{with }\lim_{j\to\infty}\lambda_j=\infty.
$$
The discrete counterpart seeks $(\lambda_h,u_h)\in\mathbb R \times V_h$
with $\|u_h\|_b=1$ such that
\begin{equation}\label{e:maxwell_EVP_discr}
 a(u_h,v_h) = \lambda_h b(u_h,v_h) 
 \quad\text{for all }v_h\in V_h.
\end{equation}

We now focus on the $k$th eigenpair $(\lambda,u)$
and its approximation $(\lambda_h,u_h)$.
Recall $M_h$ from \eqref{e:Galerkin_schranke},
which in practice is bounded by 
$\hat M_h$ from Theorem~\ref{t:Mhtheorem}.

\begin{theorem}\label{t:lowerbound}
Let $(\lambda,u)\in\mathbb R\times V$ with
$\|u\|_b=1$ be the $k$th eigenpair to
\eqref{e:maxwell_EVP} and 
 $(\lambda_h,u_h)\in \mathbb R\times V_h$ 
 with $\|u_h\|_b=1$ be the $k$th discrete
 eigenpair to \eqref{e:maxwell_EVP_discr}.
 The following lower bound holds
 $$
   \frac{\lambda_h}{1+M_h^2\lambda_h} \leq \lambda .
$$
In particular, we have the computable guaranteed lower bound 
$$
   \frac{\lambda_h}{1+\hat M_h^2\lambda_h} \leq \lambda 
$$
with $\hat M_h$ from Theorem~\ref{t:Mhtheorem}.
\end{theorem}

\begin{proof}
We denote an $a$-orthonormal set of first $k$ eigenfunctions
by $u_1,\dots u_k$.
If $k=1$, we denote the first eigenfunction by $v:=u_1$
and note that the discrete Rayleigh--Ritz principle
implies 
$\lambda_h \|G_h v\|_b^2 \leq \| G_h v\|_a^2$.
If $k\geq 2$, we consider in a first case that the space spanned by
$G_h u_1,\dots, G_h u_k$ has dimension $k$.
The discrete Rayleigh--Ritz principle
\cite{WeinsteinStenger1972} for the $k$th discrete
eigenvalue $\lambda_h$ states
$$
  \lambda_{h} =
         \min_{V_h^{(k)}\subseteq V_h} 
         \max_{v_h\in V_h^{(k)}\setminus \{0\}} 
              \frac{\|v_h\|_a^2}{\|v_h\|_b^2} 
     \leq 
     \max_{v_h\in \operatorname{span}\{G_h u_1,\dots, G_h u_k\}\setminus \{0\}} 
              \frac{\|v_h\|_a^2}{\|v_h\|_b^2} 
$$ 
where the minimum runs over all $k$-dimensional subspaces
$V_h^{(k)}$ of $V_h$.
There exist real coefficients $\xi_1,\dots,\xi_k$
with $\sum_{j=1}^k \xi_j^2 =1$ such that the maximizer
on the right-hand side equals $G_h v$ for
$v=\sum_{j=1}^k \xi_j u_j$.
Recall the Helmholtz decomposition
\eqref{e:helmholtz} with the divergence-free part $R$ of $v$.
The above Rayleigh--Ritz principle implies
$$
  \lambda_{h} \leq \frac{\|G_h v\|_a^2}{\|G_h v\|_b^2} 
$$
and therefore (for $k\geq 1$)
\begin{equation}\label{e:EVPbeweis_a}
  \lambda_h \|R\|_b^2 \leq 
\lambda_h \|G_h v\|_b^2 \leq \| G_h v\|_a^2.
\end{equation}
We expand the square on the left-hand side,
use that $\|v\|_b^2=1$, and use the Young inequality
with an arbitrary $0<\delta<1$ to infer
\begin{equation*}
\|R\|_b^2 
=
 \|R-v\|_b^2 + 1 + 2b(R-v,v) 
\geq
 (1-\delta^{-1}) \|R-v\|_b^2 + (1-\delta)
\end{equation*}
(note that $1-\delta^{-1}<0$).
The combination with Lemma~\ref{l:L2} results in
\begin{equation}\label{e:EVPbeweis_b}
 \lambda_h \|R\|_b^2 
 \geq 
 \lambda_h 
 \left(
(1-\delta^{-1}) M_h^2 \|G_h v -v\|_a^2 + (1-\delta)
 \right)  .
\end{equation}
The Galerkin orthogonality \eqref{e:galerkinorthogonality}
in the $a$-product 
and the estimate $\|v\|_a^2\leq\lambda$ for the right-hand side
of \eqref{e:EVPbeweis_a} result in
\begin{equation}\label{e:EVPbeweis_c}
  \|G_h v\|_a^2
\leq 
 \lambda -\| v-G_h v\|_a^2.
\end{equation}
The choice $\delta = (\lambda_h M_h^2)/(1+\lambda_h M_h^2)$
and the combination of
\eqref{e:EVPbeweis_a}--\eqref{e:EVPbeweis_c} results in
\begin{equation}\label{e:proof_boundMh}
   \frac{\lambda_h}{1+M_h^2\lambda_h} \leq \lambda .
\end{equation}
In the remaining case that the space spanned by
$G_h u_1,\dots, G_h u_k$ has dimension strictly less than $k$,
there exists a $b$-normalized function $v$ in the linear hull of the
functions $u_1,\dots,u_k$ such that $G_h v =0$.
The divergence-free part $R$ of $G_h v$ is thus zero
and Lemma~\ref{l:L2} implies
$$
 1=\|v\|_b^2 = \|v-R\|_b^2 \leq M_h^2 \|v-G_h v\|_a^2
  =M_h^2 \|v\|_a^2.
$$
Since $v$ is taken from the linear hull of the first $k$ eigenfunctions,
the orthogonality of the latter implies
$\|v\|_a^2\leq \lambda_k$.
Therefore $M_h^{-2}\leq \lambda_k$ (excluding the trivial case $M_h=0$),
which implies the bound \eqref{e:proof_boundMh} also in the second case.
The stated computable bound follows from \eqref{e:proof_boundMh}
$M_{h} \leq \hat M_h$ 
and the monotonicity of the left hand side
in the proven estimate.
\end{proof}

\section{Bounds on the involved constants}
\label{s:constants}

This section describes how the critical constants
are computed or computationally bounded.

\subsection{Computation of \texorpdfstring{$\kappa_{h}$}{kappah}}
In this paragraph we briefly sketch the numerical
computation of $\kappa_h$. The reasoning is similar
to \cite{LiuOishi}, and we illustrate the 
extension of their approach to the Maxwell operator.
Recall from Subsection~\ref{ss:galerkinerror}
the space $X_h$
and the set $\mathcal S_{f_h}$ for given $f_h\in X_h$.
Let furthermore $\tilde u\in V$ denote the solution
to the linear problem \eqref{e:maxwell_linear}
with right-hand side $f_h$.
Expanding squares
and integration by parts then shows for any
$\tau_h\in \mathcal S_{f_h}$ and any
$v_h\in\mathcal N_{0,D}(\mathcal T)$ the
hypercircle identity \eqref{e:hypercircle}
because $\Curl(\tau_h -\rot\tilde u) =0$.
Thus, the optimization problem
\begin{equation*}
\min_{v_h\in\mathcal N_{0,D}(\mathcal T)}
\min_{\tau_h\in \mathcal S_{f_h}}
\| \tau_h - \rot v_h\|_{L^2(\Omega)}^2
\end{equation*} 
is equivalent to
\begin{equation*}
 \min_{v_h\in\mathcal N_{0,D}(\mathcal T)}
 \| \rot (\tilde u -  v_h)\|_{L^2(\Omega)}^2
 +
\min_{\tau_h\in \mathcal S_{f_h}}
\| \tau_h - \rot \tilde u\|_{L^2(\Omega)}^2 .
\end{equation*}
The minimizer to the first term is given by the solution
to the primal problem \eqref{e:maxwell_linear_discr}
while the minimizer to the second term is given
by $\sigma_h$ as part of the solution pair 
$(\sigma_h,\rho_h)\in \mathcal  S_h\times X_h$
to the following (dual) saddle-point problem
\begin{align*}
(\sigma_h,\tau_h)_{L^2(\Omega)}
+ 
(-1)^d
(\Curl\tau_h,\rho_h)_{L^2(\Omega)}
&=0
&&
\text{for all }\tau_h\in \mathcal S_h
\\
(-1)^d
(\Curl\sigma_h,\phi_h)_{L^2(\Omega)}
&=
-(f_h,\phi_h)_{L^2(\Omega)}
&&
\text{for all }\phi_h\in X_h.
\end{align*}
This can be shown with arguments analogous to
\cite[III§9, Lemma 9.1]{Braess2007}.

Let $T_h:X_h\to\mathcal N_{0,D}(\mathcal T)$
denote the solution
operator to the primal problem
and let $S_h:X_h\to \mathcal S_h$,
$R_h:X_h\to X_h$ denote the components of the solution
operator to the dual problem.
The operators $T_{h}$ and $R_{h}$ are symmetric,
and straightforward calculations involving the definitions
of $S_h$, $T_h$, $R_h$ prove that
the error can be represented as
\begin{equation*}
\|\rot u_h - \sigma_h\|_{L^2(\Omega)}^2
= -(f_{h}, T_{h}f_{h})_{L^2(\Omega)} 
+ (f_{h}, R_{h}f_{h})_{L^2(\Omega)}.
\end{equation*}

The divergence-free constraint in the primal
problem is practically implemented in a
saddle-point fashion.
In what follows, we identity the piecewise
constant function $f_h$ with its
vector representation.
The matrix structure of the discrete problem
is as follows
\begin{equation*}
L
\left(\begin{array}{c}
w_{1} \\
w_{2} 
\end{array}\right)=
\left(\begin{array}{c}
B f_h \\
0 
\end{array}\right)
\qquad
\text{where }
L:=\left(\begin{array}{cc}
D & F^{T}  \\
F & \phantom{{-}}0 
\end{array}\right).
\end{equation*}
The solution
can be expressed by
$$ w_{1} = L^{-1}(*,*)Bf_h= HBf_h,$$
where $H:= L^{-1}(*,*)$ denotes the relevant rows and columns of $L^{-1}$ for the computation of the first component of the solution.

The dual system can be implemented by
introducing Lagrange multipliers related to the 
interior hyper-faces to enforce normal-continuity.
In terms of matrices, the dual system reads
\begin{equation*}
K
\left(\begin{array}{c}
z_{1} \\
z_{2} \\  
z_{3}
\end{array}\right)=
\left(\begin{array}{c}
0 \\
M f_h \\  
0
\end{array}\right)
\qquad
\text{where }
K:=\left(\begin{array}{ccc}
A & -G^{T} & 0 \\
G & \phantom{{-}}0 & C \\  
0 & \phantom{{-}}C^{T} & 0
\end{array}\right).
\end{equation*}
The solution can be expressed by
$$ z_{2} =  K^{-1}(**,**)Mf_h=\tilde{H}Mf_h,$$
where $\tilde{H}:= K^{-1}(**,**)$ denotes the relevant rows and columns of $K^{-1}$ for the computation of the second component of the solution.

Thus,
$$ 
\|\rot u_h -\sigma_h\|_{L^2(\Omega)}^2
 = f_h^{T} B^{T} H B f_h + f_h^{T} M^{T} \tilde{H} M f_h 
 = f_h^{T} Q f_h, 
$$
where $Q := B^{T} H B + M^{T} \tilde{H} M$. 
Then $\kappa_{h}$ 
has the following representation
$$
 \kappa_{h} 
 = \max_{f_{h} \in X_h\setminus\{0\}}
\frac{\|\rot u_h -\sigma_h\|_{L^2(\Omega)}}
     {\|f_h\|_{L^2(\Omega)}}
= 
\max_{y\in\mathbb R^m\setminus\{0\},C^{T}y=0}
\left(\frac{y^{T} Q y}{y^{T} M y}\right)^{-1/2} .
$$
The constraint $C^{T}y=0$ accounts for the fact that
the elements of $X_h$ are the piecewise constant
vector fields satisfying normal continuity across
the element boundaries.
Since
$Q$ is symmetric,  $\kappa_{h}$ is the square 
root of the maximum eigenvalue $\mu$
of the eigenvalue problem
$$ \left(\begin{array}{cc}
Q & C^{T}  \\
C & 0 
\end{array}\right) y = \mu \left(\begin{array}{cc}
M & 0  \\
0 & 0 
\end{array}\right) y .
$$ 
The eigenvalue problem can be 
numerically solved with a power method and an inner iteration. The idea of this method is similar to Lanczos method with an inner iteration \cite{Golub}.

\subsection{Bound on local Poincar\'e constants}\label{ss:poincare}

It is well known that the Poincar\'e constant over a domain
$\omega$ equals $\mu^{-1/2}$ for the first Laplace--Neumann
eigenvalue $\mu$.
We compute upper bounds on the Poincar\'e constant of element patches
by determining lower bounds on the first Neumann eigenvalue.
To this end, we use the method of 
\cite{CarstensenGedicke}
(with improved constants from \cite{CarstensenPuttkammer})
on a sub-triangulation of the patch.
In our computation, for every possible cell patch $\omega_{T}$ a Neumann eigenvalue
problem for the Laplacian is solved.
The lower bound for the first eigenvalue $\lambda_{1}$ is given by $\hat{\lambda}_{1}$
$$ 
\hat{\lambda}_{1}:=\frac{\lambda_{CR,1}}{1+\kappa^{2} \lambda_{CR,1}H^{2}} \leq \lambda_{1} 
$$
where $\kappa^{2} \approx 0.0889$ in two dimensions and 
$\kappa^{2} \approx 1.0083$ in three dimensions,
$H$ is the maximum mesh size of the sub-triangulation of
$\omega_{T}$, and $\lambda_{CR,1}$ is the first discrete eigenvalue computed 
by the Crouzeix--Raviart method. 
We remark that improved values for $\kappa$ were worked out in
\cite{Liu2015}.
To get an adequate lower bound of the Poincar\'e constant
we use a submesh generated by three uniform refinements of the patch.
We divide the resulting Poincar\'e constant 
$\operatorname{diam}(\omega_{T}) \tilde{c}$ by
$\text{diam}(\omega_{T})$ and obtain a mesh-size independent upper bound $\tilde c$,
which in our two-dimensional computations on structured meshes
takes the value $\tilde{c}=0.2461$ for the chosen sequence of
red-refined meshes.

\subsection{Computation of the projection operator constant}
\label{ss:FWcomputation}
The goal of this section is to explicitly determine the constants
$C_{1,\ddiv}$, $C_{2,\ddiv}$, $C_{1,\Curl}$, $C_{2,\Curl}$
from \eqref{e:fwdivstab} and \eqref{e:fwrotstab}
in the two-dimensional case.
We briefly review the construction \cite{FalkWint2} and
the main steps in bounding the involved
constants.
In order to stay close to the notation of
\cite{FalkWint2} and to refer to the construction
in their format,
we consider the complex built
by the operators $\nabla$ and $\rot$.
By the isometry of $\Curl$ and $\nabla$ and of
$\rot$ and $\ddiv$ in two dimensions, the results
can then be used for the 
$\Curl$-$\ddiv$ complex.
More precisely, after rotation of coordinates,
the operator $\pi^\nabla$ replaces $\pi^{\Curl}$
from Section~\ref{ss:FW}
and the operator $\pi^{\rot}$ replaces
$\pi^{\ddiv}$ from Section~\ref{ss:FW}
(with identical stability constants).

Given a triangulation $\mathcal{T}$,
we denote by
$\Delta_{0}(\mathcal{T})$ the set of all
vertices and by $\Delta_1(\mathcal T)$ the set of
all edges of $\mathcal{T}$.
For any $T\in\mathcal T$,
$\Delta_0(T)$ is the set of vertices and
$\Delta_1(T)$ is the set of edges of $T$.

\subsubsection{Falk--Winther operator for the gradient in 2D}

Given any $y \in \Delta_{0}(\mathcal{T})$, the associated
macroelement (or vertex patch) $\omega_y$
is defined as follows
$$
\omega_{y} 
:= \text{int } (\cup \lbrace{T \in \mathcal{T}: y \in T \rbrace}) .
$$
The subset of $\mathcal T$ of triangles having nonempty
intersection with $\omega_y$ is denoted
by $\mathcal T(\omega_y)$.
Analogous notation applies to other open subsets
$\omega\subseteq\Omega$.
Given $u \in H^{1}(\Omega)$,
the discrete function $\pi^\nabla u$ is given by its
expansion
$$
  \pi^\nabla u =
   \sum_{y \in \Delta_{0}(\mathcal{T})} c_y(u) \lambda_{y}
$$ 
where $\lambda_{y}$ is the piecewise linear hat function associated the vertex $y$, i.e $\lambda_{y}(y)=1$ and $\lambda_{y}=0$ on the complement of the macroelement
$\omega_{y}$.
The coefficient $c_y(u)$ is given by
$$
  c_y(u) = \fint_{\omega_y} u\,dx
     +(Q_y^0 u) (y)
$$
where $Q_y^0 u \in S^1(\mathcal T(\omega_y))$ solves
the discrete Neumann problem
$$
\begin{aligned}
 \int_{\omega_{y}} \nabla (Q_{y}^{0}u-u)\cdot \nabla v_h\, dx 
 &= 0
 \quad\text{for all } v_h\in S^1(\mathcal T(\omega_y)) ,
 \\
 \int_{\omega_y} Q_y^0u \,dx 
 &=0 .
 \end{aligned}
 $$
 We denote by $C_1(y,T)$ the constant such that
 $$
  |v_h(y)|^2 
  \leq 
  C_1(y,T) \|\nabla v_h\|_{L^2(\omega_y)}^2
 $$
 holds for all $v_h\in S^1(\mathcal T(\omega_y))$
 with $\int_{\omega_y} v_h\,dx =0$.
 It is readily verified that $C_1(y,T)$ is independent of
 the mesh size. 
For the actual computation of
$C_1(y,T)$ we introduce bilinear forms
bilinear forms
 \begin{align*}
  a(u,v) = u(y) v(y) \quad\text{and}\quad
  b(u,v) = (\nabla u, \nabla v)_{L^{2}(\omega_{y})}
 \end{align*}
 on the space $S^1(\mathcal T(\omega_y))$.
 It is direct to verify that $C_1(y,T)$ equals the largest eigenvalue
 of the generalized discrete eigenvalue problem
 which seeks 
 $(\mu,u) \in \mathbb{R}\times  S^1(\mathcal T(\omega_y))$ such that
  \begin{equation}\label{e:C1yT_problem}
   a(u,v) = \mu b(u,v) \qquad \quad\text{for all } v \in S^1(\mathcal T(\omega_y)).
  \end{equation}

  Given any $T\in\mathcal T$, we then compute
 with triangle and Young inequalities
 $$
  \| \pi^\nabla u\|_{L^2(T)}^2
  \leq
  3
  \sum_{y\in \Delta_0(T)}
   |c_y|^2 \|\lambda_y\|_{L^2(T)}^2.  
 $$
 For any vertex $\Delta_0(T)$
 we use the definition of $c_y$,
 and the definition of
 $C_1(y,T)$ to infer
 \begin{align*}
  |c_y|^2 
  &\leq
  2 
  ( |\fint_{\omega_y} u dx|^2
  +
   | (Q_y^0 u) (y) |^2
   )
  \\
  &\leq
   2\operatorname{meas}(T)^{-1} \| u \|_{L^2(\omega_y)}^2
  +
  2 C_1(y,T)
   \|\nabla Q_y^0 u\|_{L^2(\omega_y)}^2.
 \end{align*}
 The bound
$\| \nabla Q_y^0 u\|_{L^2(\omega_y)}
 \leq
 \| \nabla  u\|_{L^2(\omega_y)}
 $
 is immediate and so concludes the stability analysis.
 The norm of $\lambda_y$ satisfies
 $\|\lambda_y\|_{L^2(T)}^2=\operatorname{meas}(T)/6$.
 The resulting bound reads
\begin{align*}
  \| \pi^\nabla u\|_{L^2(T)}^2
 \leq
  3 \| u \|_{L^2(\omega_T)}^2
  +
  \frac{\operatorname{meas}(T)}{h_T^2}
  \sum_{y\in \Delta_0(T)}
   C_1(y,T) \;
  h_T^2
   \|\nabla  u\|_{L^2(\omega_T)}^2.
\end{align*}
Summarizing, we have
\begin{align*}
  \| \pi^\nabla u\|_{L^2(T)}
 \leq
  C_{1,\Curl} \| u \|_{L^2(\omega_T)}
  +
  C_{2,\Curl} h_{T}
   \|\nabla  u\|_{L^2(\omega_T)}
\end{align*}
where
\begin{align*}
C_{1,\Curl} := \sqrt{3}
\quad\text{and}\quad
C_{2,\Curl} := \sqrt{\frac{\operatorname{meas}(T)}{h_{T}^{2}} 
                       \sum_{y\in \Delta_0(T)} C_1(y,T) }.
\end{align*}

\subsubsection{Falk--Winther operator for the rotation in 2D}

The space $\mathcal{N}_0(\mathcal{T})$ is spanned by the 
edge-oriented basis functions 
$(\psi_{E})_{E \in \Delta_{1}(\mathcal{T})}$ that are uniquely 
defined for any $E \in \Delta_{1}(\mathcal{T})$ through the property
\begin{equation*}
 \int_{E} \psi_{E} \cdot t_{E} \, ds = 1 \quad \mbox{ and } \quad \int_{E'} \psi_{E} \cdot t_{E} \, ds = 0 \quad \mbox{ for all } E' \in \Delta_{1}(\mathcal{T})\setminus\{E\},
\end{equation*}
where $t_{E}$ denotes the unit tangent to the edge $E$ with a 
globally fixed sign. 
For each vertex $y \in \Delta_{0}(\mathcal{T})$,
the piecewise constant function $z_{y}^{0}$ is given by
$$ z_{y}^{0} = \left\{
\begin{array}{ll}
(\text{meas}(\omega_{y}))^{-1} & \mbox{in } \omega_{y} \\
0 & \, \textrm{else}.
\end{array}
\right. $$
The extended edge patch
of an edge
$E=\operatorname{conv}\{y_1,y_2\}$
with $y_1,y_2\in\Delta_0(\mathcal T)$
is given by 
$$
\omega_{E}^{e} := \omega_{y_{1}} \cup w_{y_{2}} .
$$
The piecewise constant function 
$(\delta z^{0})_{E} \in L^{2}(\omega_E^e)$ is given by
$$ 
(\delta z^{0})_{E} := z_{y_{1}}^{0} - z_{y_{2}}^{0}.
$$
The Falk--Winther operator $\pi^{\rot}$ is defined as
\begin{equation*}
\pi^{\rot} u = S^{1} u + \sum_{E \in \Delta_{1}(\mathcal{T})} \int_{E} ((I-S^{1})Q_E^1 u) \cdot t_{E} ds \;\psi_{E},
\end{equation*}
where
\begin{align*} 
S^{1}u &:= M^{1} u + \sum_{y \in \Delta_{0}(\mathcal{T})} (Q_{y,-}^{1} u)(y) \nabla \lambda_{y} \\
\text{and}\quad
M^{1}u &:= \sum_{E \in \Delta_{1}(\mathcal{T})} \int_{\omega_{E}^{e}} u \cdot z_{E}^{1} \, dx \psi_{E} .
\end{align*}
The definition of the involved objects
$Q_{y,-}^1$, $z_E^1$, $Q_E^1$ is as follows.

The operator 
$Q_{y,-}^{1}: H(\rot, \omega_{y}) \to S^1(\mathcal{T}(\omega_{y}))$ is given by the solution of the local discrete Neumann problem
\begin{align*}
(u - \nabla Q_{y,-}^{1} u, \nabla v)_{L^2(\omega_{y})} &= 0 
&\text{for all } v \in S^1(\mathcal{T}(\omega_{y})) \\
\int_{\omega_{y}} Q_{y,-}^{1} u \, dx &= 0 .&
\end{align*}

We denote by $\mathit{RT}_{0,D}(\mathcal{T}(\omega_E^e))$
the space of elements from $\mathit{RT}_0(\mathcal{T}(\omega_E^e))$
with vanishing normal trace on the boundary of $\omega_E^e$.
The weight function $z_{E}^{1}$ is given as the 
solution to the following saddle point problem:
Find 
$(z_{E}^{1},v) \in \mathit{RT}_ {0,D}(\mathcal{T}(\omega_E^e)) 
 \times S^1_0(\mathcal{T}(\omega_{E}^{e}))$ such that
 \begin{equation}\label{e:z1E_problem}
 \begin{aligned}
(\ddiv z_{E}^{1}, \ddiv \tau)_{L^2(\omega_E^e)} 
+ (\tau, \Curl v )_{L^2(\omega_E^e)}
&= (- (\delta z^{0})_{E}, \ddiv \tau)_{L^2(\omega_E^e)}
\\
(z_{E}^{1}, \Curl w)_{L^2(\omega_E^e)} &=
0 
\end{aligned}
\end{equation}
for all
$\tau \in \mathit{RT}_{0,D}(\mathcal{T}(\omega_{E}^{e})) $
and all
$w \in S^1_0(\mathcal{T}(\omega_{E}^{e}))$.

Given an edge $E$ and some $u\in H(\rot,\omega_{E}^e)$, 
the function
$Q_{E}^{1}(u) \in \mathcal{N}_0(\mathcal{T}(\omega_{E}^{e}))$
is defined by the system
\begin{align*}
(u - Q_{E}^{1} u, \nabla \tau)_{L^2(\omega_E^e)}
&= 0 
&& \text{for all } \tau \in \mathcal{S}(\mathcal{T}(\omega_{E}^{e})) \\
(\rot(u - Q_{E}^{1}u), \rot v)_{L^2(\omega_E^e)} &= 0
&& \text{for all } v \in \mathcal{N}_0(\mathcal{T}(\omega_{E}^{e})) .
\end{align*}

We proceed by computing upper bounds to the stability
constant. Let $T\in\mathcal T$.
The triangle inequality implies
\begin{align}
\begin{aligned}
\label{e:pirot_triang}
\|\pi^{\rot} u\|_{L^2(T)}
\leq
&
\|M^{1}u\|_{L^2 (T)} 
+ \|\sum_{y \in \Delta_{0}(T)} 
       (Q_{y,-}^{1} u)(y) \nabla \lambda_{y}\|_{L^2 (T)}
\\
&\qquad
       + \|\sum_{E \in \Delta_{1}(T)} 
  \int_{E} ((I-S^{1})Q_E^1 u) \cdot t_{E} ds
  \psi_{E}\|_{L^2 (T)} .
\end{aligned}
\end{align}
In what follows we refer to the three terms on the 
right hand side as the `first', `second', and `third'
term.

\paragraph{Bound on the first term}
Elementary estimates imply
\begin{align*}
\|M^{1}u\|_{L^2(T)}^2
&\leq 
\| \sum_{E\in \Delta_1(T)} 
 |\int_{\omega_{E}^{e}} u z_{E}^{1} \, dx| 
 |\psi_{E}| \|_{L^2(T)}^{2} \\
\\
&\leq
3 \|u\|_{L^2(\omega_{T})}^2
\sum_{E\in \Delta_1(T)}  
\|z_{E}^{1}\|_{L^2(\omega_{E}^{e})}^{2} 
\|\psi_{E}\|_{L^2(T)}^{2}  .
\end{align*} 
 With the constant
 $$ C_{M_{1}}:= \sqrt{3 \sum_{E\in \Delta_1(T)}  
 \|z_{E}^{1}\|_{L^2(\omega_{E}^{e})}^{2} 
 \|\psi_{E}\|_{L^2(T)}^{2}}
 $$
 we thus have the local bound
 $$
 \|M^{1}u\|_{L^2(T)} \leq  C_{M_{1}} \|u\|_{L^2(\omega_{T})}.
 $$

\paragraph{Bound on the second term}
A scaling argument shows that there is a mesh-size independent
constant $C_{Q,T}$ such that
\begin{equation*}
|Q_{y,-}^{1}u(y) |^{2} 
\|\nabla \lambda_{y}\|_{L^2(T)}^{2}
\leq C_{Q,T}^2 \| \nabla Q_{y,-}^{1} u\|_{L^2(\omega_{y})}^{2}.
\end{equation*}
The value of the constant $C_{Q,T}$ can be computed with a help of
 the following discrete eigenvalue problem. 
Define bilinear forms
 \begin{align*}
  a(u,v) = u(y) v(y) \|\nabla \lambda_{y}\|_{L^2(T)}^{2} 
  \quad\text{and}\quad
  b(u,v) = (\nabla u, \nabla v)_{L^{2}(\omega_{y})}.
 \end{align*}
  Then, $C_{Q,T}^2$ equals the maximal eigenvalue $\mu$ 
  with eigenfunction $u\in S^1(\mathcal T(\omega_y))$ such that
  \begin{equation}\label{e:CQT_problem}
    a(u,v) = \mu b(u,v) \quad \text{for all } v\in S^1(\mathcal T(\omega_y))
  \end{equation}

Furthermore, the stability estimate
$\| \nabla Q_{y,-}^{1} u\|_{L^2(\omega_{y})}
\leq
\|u\|_{L^2(\omega_{T})}^{2}$
is immediate from the system defining $Q_{y,-}^{1}$.
We then have
\begin{align*}
\| \sum_{y\in \Delta_0(T)} (Q_{y,-}^{1} u)(y)
  \nabla \lambda_{y} 
 \|_{L^2(T)}^{2} 
 &\leq 3 \sum_{y \in\Delta_0(T)} |(Q_{y,-}^{1} u)(y)|^{2}
 \|\nabla \lambda_{y}\|_{L^2(T)}^{2} \\
&\leq 3 \sum_{y \in \Delta_0(T)}
 C_{Q,T}^{2} \|u\|_{L^2(\omega_{y})}^{2} 
\leq 9 C_{Q,T}^{2} \|u\|_{L^2(\omega_{T})}^{2} .
\end{align*}

\paragraph{Bound on the third term}
We note that there is a constant $C_S$ such that
\begin{equation*}
|\int_{E} (I-S^{1}) Q_E^1 u \cdot t_{E} \, ds |^{2} 
 \|\psi_{E}\|_{L^2(T)}^{2} 
 \leq C_{S} \|Q_E^1  u\|_{L^2(\omega_{E}^{e})}^{2} .
\end{equation*}
The constant $C_{S}$ can be computed as the largest eigenvalue $\mu$
of the auxiliary eigenvalue
\begin{equation}\label{e:CS_problem}
\int_{E} (I-S^{1}) u_h \cdot t_{E} \, ds
 \cdot \int_{E} (I-S^{1}) v_h \cdot t_{E} \, ds
 \|\psi_{E}\|_{L^2(T)}^2
 =
 \mu
 (u_h,v_h)_{L^2(\omega_{E}^{e})}
\end{equation}
on the space $\mathcal{N}_0(\mathcal{T}(\omega_{E}^{e}))$.
In order to bound the norm of $Q_E^1  u$, we use the discrete
Helmholtz decomposition
$$
 Q_E^1  u  = \nabla \alpha_h + R,
$$
where $\alpha_h\in S^1(\mathcal T(\omega_E^{e}))$
and $R$ is a discretely divergence-free N\'ed\'elec function
with homogeneous tangential boundary conditions.
From a discrete Maxwell eigenvalue problem it follows that
\begin{equation}\label{e:local_discr_maxwell}
 \|R\|_{L^2(T)}
 \leq 
 \|R\|_{L^2(\omega_E^e)}
  \leq c_{M} h_T \| \rot R \|_{L^2(\omega_E^e)}.
\end{equation}
Moreover, from the definition of $Q_{E}^{1}$, we 
deduce the stability
$
\|\rot Q_{E}^{1} u\|_{L^2(\omega_{E}^{e})}
\leq 
\|\rot u\|_{L^2(\omega_{E}^{e})}.
$
From the definition of $Q_{E}^{1}$, we infer from testing with 
$\tau = \alpha_h$ 
$$ 
 \|\nabla \alpha_h\|_{L^2(\omega_E^e)}^{2} 
 = (u, \nabla \alpha_h)_{L^2(\omega_E^e)}
 \leq \|u\|_{L^2(\omega_E^e)}  \|\nabla \alpha_h \|_{L^2(\omega_E^e)}.
$$
This yields
\begin{equation*}
\|\nabla \alpha\|_{L^2(\omega_E^e)} \leq \|u\|_{L^2(\omega_E^e)}.
\end{equation*}
The orthogonality of the decomposition therefore 
shows
\begin{align*}
\|Q_{E}^{1} u \|_{L^2(\omega_E^e)}^{2} 
=\| R \|_{L^2(\omega_E^e)}^{2} 
      + \|\nabla \alpha_h\|_{L^2(\omega_E^e)}^{2} 
\leq 
 c_{M}^2 h_T^2 \|\rot u\|_{L^2(\omega_{E}^{e})}^2
+ \|u\|_{L^2(\omega_{E}^{e})}^{2}.
\end{align*}
Thus, we obtain 
\begin{align*}
&
\|\sum_{E \in \Delta_{1}(T)}
  \int_{E} ((I-S^{1})Q_E^1 u) \cdot t_{E} ds \psi_{E}\|_{L^2(T)}^{2}
\\
&\qquad\leq  
3 \sum_{E \in \Delta_1(T)}
|\int_{E} ((I-S^{1})Q_E^1 u) \cdot t_{E} ds|^2 
  \|\psi_{E}\|_{L^2(T)}^{2} \\
&\qquad\leq 
9  C_{S} \left((c_{M} h_T)^{2} 
   \|\rot u\|_{L^2(\omega_T)}^{2} 
    + \|u\|_{L^2(\omega_T)}^{2}\right) .
\end{align*}
 We collect the bounds for the individual terms in
 \eqref{e:pirot_triang} and conclude
 \begin{align*}
\|\pi^{\rot} u\|_{L^2(T)}
&\leq
C_{M_{1}} \|u\|_{L^2(\omega_{T})} + 3 C_{Q,T} \|u\|_{L^2(\omega_{T})} 
\\ 
&\qquad
+ 3 \sqrt{C_{S}} \left((c_{M} h_T)
   \|\rot u\|_{L^2(\omega_T)}
    + \|u\|_{L^2(\omega_T)}\right) \\
&\leq C_{1,\text{div}}  \|u\|_{L^2(\omega_{T})} + C_{2,\text{div}} h \|\rot u\|_{L^2(\omega_T)},
\end{align*}
where 
\begin{align*}
C_{1,\text{div}} := C_{M_{1}} + 3 C_{Q,T} + 3 \sqrt{C_{S}}
\quad\text{and}\quad
C_{2,\text{div}} := 3 \sqrt{C_{S}}  c_{M}  .
\end{align*}

\begin{remark}
 The discrete problems to be solved for the computation of the
 interpolation constants are the local discrete linear problem
 \eqref{e:z1E_problem}
 and the local discrete eigenvalue problems
 \eqref{e:C1yT_problem}, 
 \eqref{e:CQT_problem},
 \eqref{e:CS_problem}
 as well as the local discrete Neumann problem described in
 Subsection~\ref{ss:poincare} and the local discrete Maxwell eigenvalue problem 
 related to \eqref{e:local_discr_maxwell}.
\end{remark}

\subsection{Computation of the regular decomposition constant}
In two dimensions the value for 
the constant $C_{RD}$ from the regular
decomposition equals 1.
In three dimensions, assuming the 
domain $\Omega$ is star-shaped
with respect to a ball $B$,
estimates on the constant can be derived by
tracking the constants from \cite{Guzman}.
Sharper bounds can be expected from the solution
of corresponding eigenvalue problems
(similar to \cite{GallistlLBB1,GallistlLBB2}),
but guaranteed inclusions require some knowledge
on the distribution of the spectrum.
By tracking the constants of \cite{Guzman} and explicit
calculations we obtained
the bound $C_{RD}\leq 2947$ for the unit cube
in three dimensions.

\section{Numerical results}\label{s:num}

In this section, we present numerical results for the
two-dimensional case
(and later a simple test in three dimensions).
The Python implementation uses 
the FEniCS library \cite{fenics}
and
the realization
of the Falk--Winther operator from
\cite{HenningPerssonCode}.
The two planar domains we consider are the 
unit square $\Omega=(0,1)^2$
and the L-shaped domain
$\Omega= (-1,1)^2 \backslash ([0, 1] \times [-1,0])$.
The initial triangulations are displayed in
Figure~\ref{f:initialmeshes}.
We consider uniform mesh refinement (red-refinement).
The computational bound $\hat M_h$ for $M_{h}$ (see Subsection~\ref{ss:Mhvalues}
for actual values) is achieved with the techniques
described in Section~\ref{s:constants}.

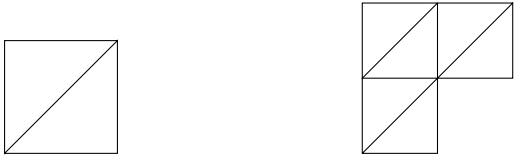
\begin{figure}
 \begin{center}
  \begin{tikzpicture}[scale=1.5] 
   \draw (0,0)--(1,0)--(1,1)--(0,0)--(0,1)--(1,1);
  \end{tikzpicture}
  \hfil
  \begin{tikzpicture} 
   \draw (0,0)--(1,0)--(1,1)--(0,0)--(0,1)--(1,1);
   \draw (0,1)--(-1,1)--(-1,0)--(0,1);
   \draw (0,0)--(-1,0)--(-1,-1)--(0,-1)--(0,0)--(-1,-1);
  \end{tikzpicture}
 \end{center}

 \caption{Initial triangulations.\label{f:initialmeshes}}
\end{figure}

\subsection{Values of the individual constants}\label{ss:Mhvalues}

We begin by reporting our upper bounds to the individual constant
entering the bound on the Falk--Winther constants.
The left column of Table~\ref{tab:constants} displays
the values of the constants 
entering the computation of the upper bound
for the unit square and the 
L-shaped domain on structured grids
with $h<\sqrt{2}/4$.
%
All these constants coincide for the two domains,
which have a similar local mesh geometry.

\begin{table}
\begin{center}
  \begin{tabular}{ l | c  }
     \text{constant} & \text{upper bound} \\ 
     \hline
           $C_{M_{1}}$ & 0.94974  \\
      $C_{Q,T}$ & 0.66666 \\
      $C_{S}$ & 2.25975 \\
      $C_{M}$ & 0.06522 \\
      $C_{1}(y,T)$ & 1.05409 \\
     $\tilde c$ & 0.2461 \\
     $C_{OL}$ & 13 \\
     $C_{RD}$ & 1 \\
     $C_{1, \Curl}$ & 1.7321  \\
     $C_{2, \Curl}$ & 0.9129  \\
     $C_{1, \ddiv}$ & 9.7290\\
  \end{tabular}
  \qquad
     \begin{tabular}{ l | c   }
      constant & upper bound  \\
      \hline
      $C_{M_{1}}$ & 1.31692  \\
      $C_{Q,T}$ & 0.66666 \\
      $C_{S}$ & 2.99264 \\
      $C_{M}$ & 0.06534 \\
      $C_{1}(y,T)$ & 1.05409 \\
      $\tilde c$ & 0.2461 \\
      $C_{OL}$ & 13 \\
      $C_{RD}$ & 1 \\
      $C_{1, \Curl}$ & 1.7321  \\
      $C_{2, \Curl}$ & 0.7394  \\
      $C_{1, \ddiv}$ & 12.29484\\ 
   \end{tabular}
\end{center}
\caption{Upper bounds of some of the relevant constants
         for the square and the L-shaped
         domain (left).
         The same quantities for an unstructured mesh (right).
         \label{tab:constants}}
\end{table}

We compare the values with results for the unstructured mesh
displayed in Figure~\ref{f:unstructured}.
The constants for this mesh geometry are displayed in 
the second column of
Table~\ref{tab:constants}. We note that there is no
significant deviation from the structured case.

 \begin{figure}
  \begin{center}
  \includegraphics[height=0.4\textwidth]{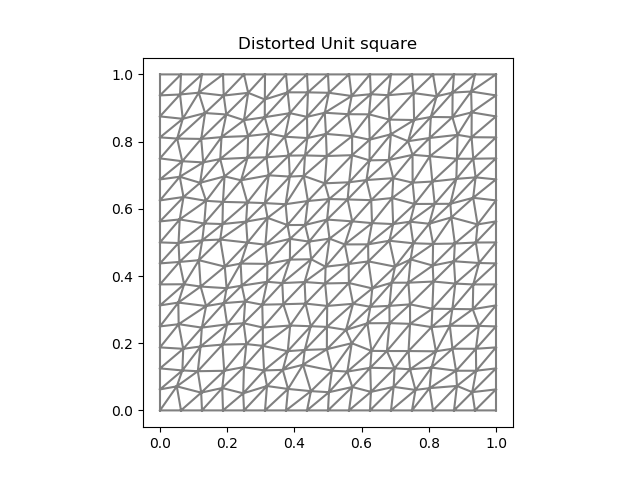} 
  \end{center} 
  \caption{An unstructured mesh
           on the square domain.\label{f:unstructured}}
 \end{figure}

\subsection{Results on the square domain}
For the square domain, it is known that the first
two eigenvalues 
are given by
$\lambda_1=\pi^{2}\approx9.8696$
and $\lambda_2=2\pi^2\approx19.7392$.
The numerical results are displayed in 
Table~\ref{tab:square_results}.
The displayed lower bounds are guaranteed, but their 
values become practically relevant after 
a moderate number of refinements only
when $M_h$ is sufficiently small.

\begin{table}[bt]
\begin{center}
  \begin{tabular}{ l | c | c || c | c || c | c  }
     $\frac{h}{\sqrt{2}}$ & $\kappa_{h}$ & $\hat M_{h}$ & $\lambda_{h}^{(1)}$ & lower bd. & $\lambda_{h}^{(2)}$ & lower bd. \\
     \hline
     $2^{-1}$ & 0.1443 & 9.1034 & 9.6000 & 0.0121 & 20.2871 & 0.0121 \\
     $2^{-2}$ & 0.0721 & 4.5499 & 9.8305 & 0.0481 & 20.0235 & 0.0482 \\
     $2^{-3}$ & 0.0356 & 2.2592 & 9.8612 & 0.1921 & 19.8205 & 0.1940 \\
     $2^{-4}$ & 0.0180 & 1.1359 & 9.8676 & 0.7186 & 19.7601 & 0.7458 \\
     $2^{-5}$ & 0.0089 & 0.5648 & 9.8691 & 2.3791 & 19.7445 & 2.7053 \\
     $2^{-6}$ & 0.0044 & 0.2806 & 9.8695 & 5.5530 & 19.7405 & 7.7269 \\
     $2^{-7}$ & 0.0023 & 0.1421 & 9.8696 & 8.2300 & 19.7395 & 14.1152 \\
     $2^{-8}$ & 0.0011 & 0.0712 & 9.8696 & 9.3992 & 19.7393 & 17.9431 \\
     $2^{-9}$ & 0.0006 & 0.0354 & 9.8696 & 9.7488 & 19.7392 & 19.2619 \\
  \end{tabular}
\end{center}
\caption{%
       Numerical results on the square domain.
       \label{tab:square_results}}
\end{table}

\begin{table}[bt]
 \begin{center}
   \begin{tabular}{ l | c | c || c | c || c | c  }
      $\frac{h}{\sqrt{2}}$ & $\lambda_{h}^{(3)}$ & lower bd. & $\lambda_{h}^{(4)}$ & lower bd. & $\lambda_{h}^{(5)}$ & lower bd. \\
      \hline
      $2^{-1}$ & 48 &  0.0121 & 57.6 & 0.0121 & 75.7128 & 0.0121 \\
      $2^{-2}$ & 36.8522 & 0.0482 & 46.7216 &  0.0483 & 78.5482 & 0.0483 \\
      $2^{-3}$ & 38.8122 & 0.1949 & 48.6686 & 0.1951 & 79.9595 & 0.1954 \\
      $2^{-4}$ & 39.3105 & 0.7600 & 49.1763 & 0.7630 & 79.2745 & 0.7675 \\
      $2^{-5}$ & 39.4362 & 2.9040 & 49.3049 & 2.9475 & 79.0401 & 3.0153 \\
      $2^{-6}$ & 39.4679 & 9.6063 & 49.3372 & 10.0980 & 78.9779 & 10.9382 \\
      $2^{-7}$ & 39.4758 & 21.9696 & 49.3453 & 24.7683 & 78.9621 & 30.4415 \\
      $2^{-8}$ & 39.4778 & 32.8925 & 49.3473 & 39.4697 & 78.9581 & 56.3816 \\
      $2^{-9}$ & 39.4782 & 37.6140 & 49.3482 & 46.4694 & 78.9572 & 71.8367 \\
   \end{tabular}
 \end{center}
 \caption{%
        Numerical results on the square domain for higher eigenvalues.
        \label{tab:square_results1}}
\end{table}

\subsection{Results on the L-shaped domain}

On the L-shaped domain, we use the 
reference value
$\lambda=1.4756218241$
from \cite{CostabelDaugeMartinVial}
for the first eigenvalue
for comparison.
The numerical results are displayed in 
Table~\ref{tab:Lshape_results} for eigenvalues $\lambda_1$ and $\lambda_2$
and in Table~\ref{tab:Lshape_results_higher} for eigenvalues
$\lambda_3$ and $\lambda_4$.
As in the previous example, the lower bounds take
values of practical significance after a couple
of refinement steps.
\begin{table}[bt]
\begin{center}
  \begin{tabular}{ l | c | c || c | c || c | c }
     $\frac{h}{\sqrt{2}}$ & $\kappa_{h}$ & $\hat M_{h}$ & $\lambda_{h}^{(1)}$ & \text{lower bd.} & $\lambda_{h}^{(2)}$ & \text{lower bd.} \\ 
     \hline

      $2^{-1}$ & 0.1355 & 8.7947 & 1.3180 & 0.0128 & 3.5356 & 0.0129\\
      $2^{-2}$ & 0.0709 & 4.5078 & 1.4157 & 0.0476 & 3.5309 & 0.0485 \\
      $2^{-3}$ & 0.0356 & 2.2592 & 1.4526 & 0.1726 & 3.5327 & 0.1856 \\
      $2^{-4}$ & 0.0180 & 1.1366 & 1.4667 & 0.5067 & 3.5336 & 0.6350 \\
      $2^{-5}$ & 0.0090 & 0.5682 & 1.4721 & 0.9979 & 3.5339 & 1.6506 \\
      $2^{-6}$ & 0.0045 & 0.2835 & 1.4743 & 1.3182 & 3.5340 & 2.7525 \\
      $2^{-7}$ & 0.0022 & 0.1421 & 1.4751 & 1.4324 & 3.5340 & 3.2987 \\
      $2^{-8}$ & 0.0011 & 0.0712 & 1.4754 & 1.4643 & 3.5340 & 3.4718 \\ 
  \end{tabular}
\end{center}
\caption{Numerical results on the L-shaped domain.
       \label{tab:Lshape_results}}
\end{table}

 \begin{table}[bt]
 \begin{center}
   \begin{tabular}{ l | c | c || c | c }
      $\frac{h}{\sqrt{2}}$ & $\lambda_{h}^{(3)}$ & \text{lower bd.} & $\lambda_{h}^{(4)}$ & \text{lower bd.} \\ 
      \hline
      $2^{-1}$ & 9.1672 & 0.0129 & 11.4797 & 0.0129 \\
      $2^{-2}$ & 9.6992 & 0.0490 & 11.3247 & 0.0490 \\
      $2^{-3}$ & 9.8272 & 0.1921 & 11.3736 & 0.1926 \\
      $2^{-4}$ & 9.8590 & 0.7177 & 11.3855 & 0.7248 \\
      $2^{-5}$ & 9.8670 & 2.3574 & 11.3885 & 2.4351 \\
      $2^{-6}$ & 9.8689 & 5.5044 & 11.3892 & 5.9472 \\
      $2^{-7}$ & 9.8694 & 8.2299 & 11.3894 & 9.2604 \\
      $2^{-8}$ & 9.8696 & 9.3992 & 11.3895 & 10.7676 \\
   \end{tabular}
 \end{center}
 \caption{Numerical results on the L-shaped domain for higher eigenvalues.
        \label{tab:Lshape_results_higher}}
 \end{table}

 \subsection{Results on the cube domain}
  We finally present a numerical test in three space dimensions
  on the cube $\Omega=(0,1)^3$.
  For our three-dimensional results we derive upper bounds on the 
  Falk--Winther operator in a fashion analogous to the computations
  of Section~\ref{s:constants}.
  We do not give a detailed account of these calculations because
  the general reasoning with the use of trace and inverse
  inequalities and the solution of local discrete eigenvalue problems is
  not different from the two-dimensional case.
  For the cube domain, it is known that the first
 eigenvalue  is given by
 $\lambda_1=2 \pi^{2}\approx19.7392$.
 The numerical results are displayed in 
 Table~\ref{tab:cube_results}.
 This table shows that our constants are too large in order to get practically relevant bounds in three dimensions.
 The results should therefore rather be seen as a proof of concept.
 Practically relevant bounds would require advanced computational techniques
 beyond our FEniCS implementation 
 to achieve a finer spatial resolution
 or alternative interpolation operators with
 sharper stability bounds.
 
 \begin{table}[bt]
 \begin{center}
   \begin{tabular}{ l | c | c || c | c  }
      $\frac{h}{\sqrt{3}}$ & $\kappa_{h}$ & $\hat M_{h}$ & $\lambda_{h}^{(1)}$ & lower bd. \\
      \hline
      $2^{-1}$ & 0.1226 & 168401.5 & 20.7306 & 3.53e-11 \\
      $2^{-2}$ & 0.0992 & 84219.0 & 20.0256 & 1.41e-10 \\
      $2^{-3}$ & 0.0513 & 42110.3 & 19.8221 & 5.64e-10 \\
      $2^{-4}$ & 0.0267 & 21055.7 & 19.7600 & 2.26e-9  \\
    
   \end{tabular}
 \end{center}
 \caption{%
        Numerical results on the cube domain.
        \label{tab:cube_results}}
 \end{table}
 
 \begin{table}
 \begin{center}
   \begin{tabular}{ l | c  }
      \text{constant} & \text{upper bound} \\ 
      \hline
      $\tilde c$ & 0.2674 \\
      $C_{OL}$ & 71 \\
      $C_{RD}$ & 2947 \\
      $C_{1, \Curl}$ & 19.7003  \\
      $C_{2, \Curl}$ & 4.5713  \\
      $C_{1, \ddiv}$ & 57.1595 \\
   \end{tabular}
 \end{center}
 \caption{Upper bounds of some of the relevant constants
          for the cube domain.
          \label{tab:constants3d}}
 \end{table}

\section{Conclusive remarks}\label{s:conclusion}

The theory on computational lower bounds to the 
Maxwell eigenvalues applies to the case of two
or three space dimensions. The sharpness of the
bounds critically depends on the actual value
of $\hat M_h$ on coarse meshes.
We remark that in two dimensions the eigenvalues
coincide with the Laplace--Neumann eigenvalues,
so our computations should be rather seen as
a proof of concept. We succeeded in bounding
$M_h$ by $\hat M_h$ such that meaningful lower bounds could
be achieved on moderately fine meshes.
The novel ingredient is the explicit computational
stability bound on the Falk--Winther operator
that makes a full quantification of the 
Galerkin error possible.
Furthermore, the methodology does not immediately
generalize to adaptive meshes in the sense that 
$\kappa_h$ is expected to scale like 
(some power of) the maximum mesh size.
The practical use of the bounds provided by the method in three
dimensions is very limited. Tighter interpolation bounds,
a refined estimate of the regular decomposition constant
$C_{RD}$, 
and the
combination with iterative solvers will be the subject of future
research.

\section*{Acknowledgments}
The first author is supported by the 
European Research Council
(ERC Starting Grant \emph{DAFNE}, agreement ID 891734).

\bibliographystyle{siamplain}
\bibliography{literatur}
\end{document}